\documentclass[runningheads,a4paper]{llncs}

\usepackage{amssymb}
\usepackage{graphicx}
\usepackage{color}

\usepackage{url}
\urldef{\mailsa}\path|cbuesing@math.tu-berlin.de|
\urldef{\mailsb}\path|d.andreagiovanni@zib.de|
\newcommand{\keywords}[1]{\par\addvspace\baselineskip
\noindent\keywordname\enspace\ignorespaces#1}

\begin{document}

\mainmatter  

\title{New results about multi-band uncertainty in Robust Optimization%
\thanks{This work was partially supported by the \emph{German Federal Ministry of Education and Research} (BMBF), project \emph{ROBUKOM} \cite{BlDAHa11,KoEtAl12}, grant 03MS616E, and by the DFG Research Center \textsc{Matheon} - \textit{Mathematics for key technologies}, project B23 - \emph{Robust optimization for network applications}.
The present paper is a revised version of the one appeared in the Proceedings of SEA 2012 \cite{BuDA12a}. The original
publication is available at www.springerlink.com.}%
}

\titlerunning{New results about multi-band uncertainty in Robust Optimization}

%
%
\author{Christina B\"using\inst{1} \and Fabio D'Andreagiovanni\inst{2}%
}
\authorrunning{C. B\"using \and F. D'Andreagiovanni}

\institute{Institut f\"ur Mathematik, Technische Universit\"at Berlin\\
Strasse des 17 Juni 136, 10623 Berlin, Germany\\
\mailsa\\
\
\\
\and
Konrad-Zuse-Zentrum f\"ur Informationstechnik Berlin (ZIB)\\
Takustrasse 7, 14195 Berlin, Germany\\
\mailsb\\
}

%
%

\toctitle{Lecture Notes in Computer Science}
\tocauthor{Authors' Instructions}
\maketitle

\begin{abstract}
``The Price of Robustness'' by Bertsimas and Sim \cite{BeSi04} represented a breakthrough in the development of a tractable robust counterpart of Linear Programming Problems.
However, the central modeling assumption that the deviation band of each uncertain
parameter is single may be too limitative in practice: experience indeed suggests that the deviations distribute also internally to the single band, so that getting a higher resolution by partitioning the band into multiple sub-bands seems advisable.
The critical aim of our work is to close the knowledge gap about the adoption of a multi-band uncertainty set in Robust Optimization: a general definition and intensive theoretical study of a multi-band model are actually still missing.
Our new developments have been also strongly inspired and encouraged by our industrial partners, which have been interested in getting a better modeling of arbitrary distributions, built on historical data of the uncertainty affecting the considered real-world problems.
In this paper, we study the robust counterpart of a Linear Programming Problem with uncertain coefficient matrix, when a multi-band uncertainty set is considered. We first show that the robust counterpart corresponds to a compact LP formulation. Then we investigate the problem of separating cuts imposing robustness and we show that the separation can be efficiently operated by solving a min-cost flow problem.
Finally, we test the performance of our new approach to Robust Optimization on realistic instances of a Wireless Network Design Problem subject to uncertainty.

\end{abstract}

\section{Introduction} \label{sec:intro}

A fundamental assumption in classical optimization is that all data are exact. However, many real-world problems involve data that are uncertain or not known with precision, because of erroneous measurements or adoptions of approximated numerical representations. If such uncertainty is neglected, optimal solutions computed for nominal data values may become costly or infeasible. As a consequence, considering the impact of uncertainty on an optimization model is a critical issue when dealing with real-world problems.

During the last years, Robust Optimization (RO) has become a valid methodology to deal with optimization problems subject to uncertainty. A key concept of RO is to model uncertainty as hard constraints, that are added to the original formulation of the problem. This restricts the set of feasible solutions to \emph{robust solutions}, i.e. solutions that are protected from deviations of the data.
Such a robust approach is crucial when dealing with high risk events, such as aircraft scheduling \cite{MuVaZe95}, or sensor placement in contaminant warning systems for water distribution networks \cite{WaHaMu06}. In such settings, standard approaches like deterministic optimization or Stochastic Programming fail to protect against severe deviations, leading to unpredictable consequences.
For an exhaustive introduction to the theory and applications of RO, we refer the reader to the book by Ben-Tal et al. \cite{BeElNe09}, to the recent survey by Bertsimas et al. \cite{BeBrCa11} and to the Ph.D. Thesis \cite{Bu11}.

An approach to model uncertain data that has attracted a lot of attention
is the so called $\Gamma$-scenario set, introduced by Betsimas and Sim (BS) \cite{BeSi04}
and then adapted to several applications.
The uncertainty model for a Linear Program (LP) considered in BS
assumes that, for each coefficient $a$ we are given a nominal
value $\bar{a}$ and a maximum deviation $d$ and that the
actual value lies in the interval $[\bar{a}-d,\bar{a}+d]$.
Moreover, a parameter $\Gamma$ is introduced to represent
the maximum number of coefficients that deviate
from their nominal value. Hence, $\Gamma$ controls the
conservativeness of the robust model and
its introduction comes from the natural observation
that it is unlikely that all coefficients deviate from their nominal value at the same time.
A central result presented in BS is that, under the previous characterization
of the uncertainty set, the robust counterpart of an LP corresponds to
a linear formulation.
This counterpart has the desirable
properties of being \emph{purely linear} and, above all, \emph{compact},
i.e. the number of variables and constraints is polynomial in the
size of the input of the deterministic problem.

The use of a single deviation band may greatly limit the power of modeling uncertainty. This is particularly evident when the probability of deviation sensibly varies within the band: in this case, neglecting the inner-band behaviour and just considering the extreme values like in BS may lead to a rough estimate of the deviations and thus to unrealistic uncertainty set, which either overestimate or underestimate the overall deviation.
Having a higher modeling resolution would therefore be very desirable. This can be accomplished by breaking the single band into multiple and narrower bands, each with its own $\Gamma$.
Such model is particularly attractive when historical data about the deviations are available, a very common case in real-world problems. A multi-band uncertainty set can indeed effectively approximate the shape of the distribution of deviations built on past observations, guaranteeing a much higher modeling power than BS.
This observation was first captured by Bienstock and taken into account
to develop an RO framework for the special case of Portfolio Optimization \cite{Bi07}. It was then extended to Wireless Network Design \cite{BiDA09,DA10}. Yet, no definition and intensive theoretical study of a more general multi-band model applicable in other contexts
have been done. The main goal of this paper is to close such gap.

We remark that investigating the adoption of a multi-band uncertainty set and studying the theoretical properties of the resulting model have been discussed with and strongly encouraged by our industrial partners, such as British Telecom Italia (BT) and Nokia Siemens Networks (NSN), in past and present collaborations about real network design \cite{BiDA09,BlDAHa11,DA10,KoEtAl12}. NSN, in particular, has been particularly interested in finding refined models for taking into account the arbitrary and non-symmetrical distributions of traffic uncertainty characterizing nation-wide optical networks \cite{Sc11}. Our modeling and theoretical developments are thus also strongly based on practical industrial needs. A better modeling of the traffic uncertainty affecting telecommunication networks has been also one of the critical objectives of our research activities in the German industrial research project ROBUKOM, aimed at developing new models and algorithms for the design of robust and survivable networks \cite{BlDAHa11,KoEtAl12}, in collaboration with NSN \cite{NSN12} and the German National Research and Education Network (DFN) \cite{DFN12}.

\subsubsection*{Contributions and Outline.}
In this work, we study the robust counterpart of an LP with uncertain coefficient matrix, when a \emph{multi-band uncertainty set} is considered. The main original contributions are:
\begin{enumerate}
  \item a compact formulation for the robust counterpart of an LP;
  \item an efficient method for the separation of robustness cuts (i.e., cuts that impose robustness), based on solving a min-cost flow instance;
  \item computational experiments comparing the performance of solving the compact formulation versus a cutting plane approach on realistic wireless network design instances.
\end{enumerate}

\noindent
In Section~\ref{sec:compactMB}, we show that the robust counterpart of an LP under multi-band uncertainty corresponds to a compact Linear Programming formulation.
We then proceed to study the separation problem of robustness cuts in Section~\ref{sec:rocuts}.
Finally, in Section \ref{sec:compStudy}, we test the performance of our new model and solution methods to Robust Optimization, to tackle the uncertainty affecting signal propagation in a set of realistic DVB-T instances of a wireless network design problem.

\subsection{Model and Notation} \label{subsec:model}

In our study about the robust counterpart of Linear Programming
Problems subject to multi-band uncertainty,
we refer to a generic Linear Program of the form:
\begin{eqnarray*}
\max &  & \sum_{j\in J}c_{j}\hspace{0.1cm}x_{j}\hspace{3.4cm}(LPP)\\
 &  & \sum_{j\in J}a_{ij}\hspace{0.1cm}x_{j}\leq b_{i}\hspace{1cm}i\in I\\
 &  & x_{j}\geq0\hspace{2.1cm}j\in J
\end{eqnarray*}

\noindent
where $I=\{1,\ldots,m\}$ and $J=\{1,\ldots,n\}$ denote the set
of constraint and variable indices, respectively.

We assume that the value of each coefficient $a_{ij}$ is uncertain and is equal to the summation of a \emph{nominal value} $\overline{a}_{ij}$ and a \emph{deviation} lying in the range $[d_{ij}^{K^{-}},d_{ij}^{K^{+}}]$, where $d_{ij}^{K^{-}},d_{ij}^{K^{+}} \in \mathbb{R}$ represent the maximum negative and positive deviations from $\overline{a}_{ij}$, respectively.
The \emph{actual value} $a_{ij}$ thus lies in the interval
$[\bar{a}_{ij}+d_{ij}^{K^{-}}, \hspace{0.1cm} \overline{a}_{ij}+d_{ij}^{K^{+}}]$.
We note that assuming that the uncertainty \emph{only affects} the elements of coefficient matrix does not limit the generality of our study, as uncertainty on the cost $c$ and on the r.h.s. $b$ can be included in a very straightforward way in the coefficient matrix \cite{BeBrCa11}.

We derive a generalization of the Bertsimas-Sim uncertainty model
by partitioning the single deviation band $[d_{ij}^{K-},d_{ij}^{K+}]$ of each coefficient $a_{ij}$ into $K$ bands, defined on the basis of $K$ deviation values:
\[
-\infty<
{d_{ij}^{K^{-}}<\cdots<d_{ij}^{-2}<d_{ij}^{-1}
\hspace{0.1cm}<\hspace{0.2cm}d_{ij}^{0}=0\hspace{0.2cm}<\hspace{0.1cm}
d_{ij}^{1}<d_{ij}^{2}<\cdots<d_{ij}^{K^{+}}}
<+\infty .
\]

\noindent
Through these deviation values, we define:
1) a set
of positive deviation bands, such that each band $k\in \{1,\ldots,K^{+}\}$
corresponds to the range $(d_{ij}^{k-1},d_{ij}^{k}]$;
2) a set of negative deviation bands, such that each band $k \in \{K^{-}+1,\ldots,-1,0\}$
corresponds to the range $(d_{ij}^{k-1},d_{ij}^{k}]$ and band $k = K^{-}$
corresponds to the single value $d_{ij}^{K^{-}}$
(the interval of each band but $k =K^{-}$ is thus open on the left).
With a slight abuse of notation, in what follows we indicate
a generic deviation band through the index $k$, with  $k \in K = \{K^{-},\ldots,-1,0,1,\ldots,K^{+}\}$
and the corresponding range by $(d_{ij}^{k-1}, d_{ij}^{k}]$.

Additionally, for each band $k\in K$, we define a lower bound $l_{k}$ and
an upper bound $u_{k}$ on the number of deviations that may fall in $k$, with $l_k, u_{k} \in \mathbb{Z}$ satisfying $0 \leq l_k \leq u_{k} \leq n$. In the case of band k = 0,
we assume that $u_{0}=n$, i.e. we
do not limit the number of coefficients that take
their nominal value. Furthermore, we assume that $\sum_{k \in K} l_k \leq n$, so that there always exists a feasible realization of the coefficient matrix.

We remark that, in order to avoid an overload of the notation, we assume that the number of bands $K$ and the bounds $l_k, u_k$ are the same for each constraint $i \in I$. Anyway, it is straightforward to modify all presented results to take into account different values of those parameters for each constraint.

All the elements introduced above define what we call a \emph{multi-band uncertainty set} $\mathcal{S}_M$.
%
%
%
We now proceed to study the robust counterpart of (LPP) under multi-band uncertainty.

\section{A Compact Robust LP Counterpart} \label{sec:compactMB}

The robust counterpart of (LPP) under a multi-band uncertainty set defined by $\mathcal{S}_{M}$ can be equivalently written as:
\begin{eqnarray*} \label{robustLP}
\max && \sum_{j \in J} c_{j} \hspace{0.1cm} x_{j}
\\
&& \sum_{j \in J} \bar{a}_{ij} \hspace{0.1cm} x_{j}
+ DEV_i(x,\mathcal{S}_{M})
\leq b_i
\hspace{1cm}
i \in I
\\
&& x_j \geq 0
\hspace{4.6cm} j \in J ,
\end{eqnarray*}

\noindent
where $DEV_i(x,\mathcal{S}_{M})$ is the maximum overall deviation allowed by the multi-band uncertainty set for a feasible solution $x$ when constraint $i$ is considered. Note that we replace the actual value of a coefficient $a_{ij}$ with the summation of the nominal value $\bar{a}_{ij}$ and a deviation
falling in exactly one of the $K$ bands.
The computation of $DEV_i(x,\mathcal{S}_{M})$ corresponds to the optimal value of the  following pure 0-1 Linear Program (note that in this case the index $i$ is fixed):
\begin{eqnarray} \label{DEV01}
DEV_i(x,\mathcal{S}_{M}) = &\max&
\sum_{j \in J} \sum_{k \in K} d^{k}_{ij} \hspace{0.1cm} x_{j} \hspace{0.1cm} y_{ij}^{k}
\hspace{2cm} (DEV01)
\nonumber
\\
&&
l_k
\hspace{0.1cm} \leq \hspace{0.1cm}
\sum_{j \in J} y_{ij}^{k}
\hspace{0.1cm} \leq \hspace{0.1cm}
u_k
\hspace{0.7cm}
k \in K
\label{DEV01_bounds}
\\
&&
\sum_{k \in K} y_{ij}^{k} \leq 1
\hspace{2cm}
j \in J
\label{DEV01_singleBand}
\\
&&
y_{ij}^{k} \in \{0,1\}
\hspace{2cm}
j \in J, k \in K .
\label{DEV01_bandVars}
\end{eqnarray}

\noindent
The binary variables $y_{ij}^k$  indicate if the deviation of a coefficient ${a}_{ij}$ lies in band $k$. The constraints (\ref{DEV01_singleBand}) ensure that each coefficient deviates in at most one band (actually these should be equality constraints, but, for assumption $u_{0}=n$ made in Section \ref{subsec:model}, we can consider inequalities). Finally, the constraints (\ref{DEV01_bounds}) impose the upper and lower bounds on the number of  deviations falling in each band $k$. Thus, an optimal solution of (DEV01) defines a distribution of the coefficients  among the bands that maximizes the deviation w.r.t.~the nominal values, while respecting the bounds on the number of deviations of each band.

\medskip

\noindent
We now prove that the robust counterpart of (LPP) under a multi-band uncertainty set $\mathcal{S}_{M}$ can be reformulated as a \emph{compact} Linear Program. To this end, consider first the linear relaxation of (DEV01):
\begin{eqnarray}\label{DEV01-relax}
&\max& \sum_{j\in J}\sum_{k\in K}d_{ij}^{k} \hspace{0.1cm} x_{j} \hspace{0.1cm} y_{ij}^{k}
\hspace{2.0cm} \mbox{(DEV01-RELAX)}
\nonumber
\\
&&
l_{k} \leq \sum_{j\in J} y_{ij}^{k} \leq u_{k}
\hspace{2.3cm}
k\in K
\label{DEV01-relax-cnstr1}
\\
&&
\sum_{k\in K}y_{ij}^{k} \leq 1
\hspace{3.1cm} j\in J
\label{DEV01-relax-cnstr2}
\\
&&
y_{ij}^{k} \geq 0
\hspace{3.7cm} j\in J, k\in K ,
\label{DEV01-relax-cnstr3}
\end{eqnarray}

\noindent
where we dropped constraints $y_{ij}^{k} \leq 1$ since they are dominated by constraints (\ref{DEV01-relax-cnstr2}).

\begin{proposition}
The polytope described by the constraints of (DEV01-RELAX) is integral.
\end{proposition}
\begin{proof}
We start by rewriting  all the constraints of (DEV01-RELAX) into the form $\alpha^T y \leq \beta$ obtaining  the following matrix form:
\[
D_i
\hspace{0.1cm}
y_i
\hspace{0.1cm}
=
\hspace{0.1cm}
{\small
    \left(\begin{array}{c|c|c|c}
    &  & \\
    -I & -I & \mbox{ } \cdots \mbox{ } & -I\\
    &  &\\
    \hline
    &  & \\
    I & I & \mbox{ } \cdots \mbox{ } & I\\
    &  &\\
    \hline
    1 \cdots 1 & & \\
     & 1\cdots1 & &\\
     & & \mbox{ } \ddots \mbox{ }  &\\
     & &  & 1\cdots1 \\
    \end{array}\right)
\hspace{0.2cm}
    \left(\begin{array}{c}
    y_{i1}^{K^{-}}\\
    \vdots\\
    y_{i1}^{K^{+}}\\
    \hline
    \vdots\\
    y_{ij}^{k}\\
    \vdots\\
    \hline
    y_{in}^{K^{-}}\\
    \vdots\\
    y_{in}^{K^{+}}\\
    \end{array}\right)
\hspace{0.1cm}
\leq
\hspace{0.1cm}
    \left(\begin{array}{c}
    \vdots\\
    - l_k\\
    \vdots\\
    \hline
    \vdots\\
    u_k\\
    \vdots\\
    \hline
    \vdots\\
    1\\
    \vdots\\
    \end{array}\right)
    }
\hspace{0.1cm}
=
\hspace{0.1cm}
g_i .
\]

\noindent
Consider now the submatrix $\tilde{D}_i$  obtained from $D_i$ by eliminating the top layer of blocks $(-I|-I|\cdots|-I)$.
It is easy to verify that $\tilde{D}_i$ is the incidence matrix of a bipartite graph: the elements of the two disjoint set of nodes of the graph are in correspondence with the rows of the two distinct layers of blocks in $\tilde{D}_i$. Moreover, every column has exactly two elements that are not equal to zero, one in the upper layer  and one in the lower layer.  Being the incidence matrix of a bipartite graph, $\tilde{D}_i$ is a totally unimodular matrix \cite{NeWo88}.

In order to show that also the original matrix $D_i$ is totally unimodular, we first need to recall the equivalence of the following three statements \cite{NeWo88}:
  1) $A$ is a totally unimodular matrix;
  2) a matrix obtained by duplicating rows of $A$ is totally unimodular;
  3) a matrix obtained by multiplying a row of $A$ by -1 is totally unimodular.
Since $D_i$ can be obtained from $\tilde{D}_i$ by duplicating each row of the upper block, and multiplying each row of the duplicated block by -1, $D_i$ is totally unimodular.

As $D_i$ is totally unimodular and the vector $g_i$ is integral, it is well-known that the polytope defined by $D_i y_i \leq g_i$ and $y_i \geq 0$ is integral, thus completing the proof.
\qed
\end{proof}

\noindent
Thanks to this integrality property, we can exploit strong duality to prove the main result of this section.

\newpage

\begin{theorem} \label{th:compact_RLP}
The robust counterpart of (LPP) under a multi-band uncertainty set $\mathcal{S}_{M}$ is equivalent to the following compact Linear Program:
\begin{eqnarray*}\label{robustCompactLP}
\max && \sum_{j\in J} c_{j} \hspace{0.1cm} x_{j}
\hspace{9.1cm}  (RLP)
\\
&& \sum_{j\in J} \bar{a}_{ij} \hspace{0.1cm} x_{j}
- \sum_{k\in K} l_{k} \hspace{0.1cm} v_{i}^{k}
+ \sum_{k\in K} u_{k} \hspace{0.1cm} w_{i}^{k}
+ \sum_{j\in J} z_{i}^{j} \leq b_{i}
\hspace{1.1cm} i\in I
\\
&& -v_{i}^{k} + w_{i}^{k} + z_{i}^{j} \geq d_{ij}^{k} \hspace{0.1cm} x_{j}
\hspace{4.8cm} i \in I, j\in J, k\in K
\\
&& v_{i}^{k}, \hspace{0.1cm} w_{i}^{k} \geq 0
\hspace{6.7cm} i \in I, k\in K
\\
&& z_{i}^{j} \geq 0
\hspace{7.4cm} i \in I,j\in J
\\
&& x_{j} \geq 0
\hspace{7.3cm} j\in J .
\end{eqnarray*}
\end{theorem}

\begin{proof}
As first step, consider the dual problem of (DEV01-RELAX):
\begin{eqnarray*}\label{dual_DEV_01}
\min &&
\sum_{k\in K}-l_{k} \hspace{0.1cm} v_{i}^{k} + \sum_{k\in K} u_{k} \hspace{0.1cm} w_{i}^{k} + \sum_{j\in J} z_{i}^{j}
\hspace{2.0cm}  \mbox{(DEV01-RELAX-DUAL)}
\\
&& -v_{i}^{k} + w_{i}^{k} + z_{i}^{j} \geq d_{ij}^{k} \hspace{0.1cm} x_{j}
\hspace{3.5cm} j\in J, k\in K
\\
&& v_{i}^{k}, \hspace{0.2cm} w_{i}^{k} \geq 0
\hspace{5.3cm} k \in K
\\
&& z_{i}^{j} \geq 0
\hspace{6.1cm} j\in J ,
\end{eqnarray*}

\noindent
where the dual variables $v_{i}^{k}, w_{i}^{k}, z_{i}^{j}$ are respectively associated with the primal constraints (\ref{DEV01-relax-cnstr1},
\ref{DEV01-relax-cnstr2},
\ref{DEV01-relax-cnstr3}) of (DEV01-RELAX) defined for constraint $i$.

Since (DEV01-RELAX) is feasible and bounded by definition of the scenario
uncertainty set $\mathcal{S}_{M}$, also (DEV01-RELAX-DUAL) is feasible and bounded and the optimal values of the two problems are the same (strong duality). Then, by Proposition
1, we can replace the maximum total deviation $DEV_i(x,\mathcal{S}_{M})$ with problem
(DEV01-RELAX-DUAL), obtaining the compact Linear Program
(RLP). This concludes the proof.
\qed
\end{proof}

\noindent
In comparison to (LPP),
this compact formulation uses \mbox{$2 \cdot K \cdot m + n \cdot m$} additional variables and includes $K \cdot n \cdot m$ additional constraints. Similar reasonings can be also done to derive a compact robust counterpart of a Mixed-Integer Linear Program with uncertain coefficient matrix.

\section{Separation of Robustness Cuts}\label{sec:rocuts}

In this section, we consider the problem of testing whether
a solution $x\in\mathbb{R}^{n}$ is robust feasible, i.e. $\sum_{j \in J} \bar{a}_{ij} \hspace{0.1cm} x_{j} + DEV_i(x,\mathcal{S}_{M}) \leq b_i$
for every scenario $S\in\mathcal{S}_{M}$ and $i\in I$. This problems becomes important if we adopt a cutting plane approach instead of directly solving the compact robust counterpart (RLP): we start by solving the nominal problem (LPP) and then we check if the optimal solution is robust. If not, we generate a cut that imposes robustness (\emph{robustness cut}) and we add it to the problem. This initial step can then be iterated as in a typical cutting plane method \cite{NeWo88}.

In the case of the Bertsimas-Sim model, the problem of separating a robustness cut
is extremely simple \cite{FiMo08}: given a solution $x\in\mathbb{R}^{n}$, for each constraint $i \in I$, the problem consists of sorting the deviations $d_{ij}^{K+}x_{j}$ in non-increasing order and choose the highest $\Gamma_i$ deviations. If for some $i$ the sum of these deviations exceeds $b_{i} - \sum_{j \in J} \bar{a}_{ij} x_{j}$ then we have found a robustness cut to add. Otherwise, $x$ is feasible and robust.

In the case of multi-band uncertainty, this simple approach does not guarantee the robustness of a computed solution.
However, we prove that for a given
solution $x\in\mathbb{R}^{n}$ and a constraint $i\in I$, checking
the robust feasibility of $x$ corresponds to solving a \emph{min-cost flow problem} \cite{AhMaOr93} and thus can be done efficiently.
The min-cost flow instance associated with the robustness check is denoted by $(G,c)_{x}^{i}$ and defined as follows.
$G$ is a directed graph that contains one vertex $v_j$ for each variable index $j \in J$, one vertex $w_k$ for each band $k \in K$ and two vertices $s,t$ that are the source and the sink of the flow. So the set of vertices is $V=\bigcup_{j\in J} \{v_{j}\} \cup \bigcup_{k\in K}\{w_{k}\} \cup \{s,t\}$. The set of arcs $A$ is the union of three sets $A_1, A_2, A_3$. $A_1$ contains one arc from $s$ to every variable vertex $v_{j}$, i.e. $A_{1}=\{(s,v_{j}): j\in J\}$.
$A_2$ contains one arc from every variable vertex $v_{j}$ to every band vertex $w_{k}$, i.e. $A_{2}=\{(v_{j},w_{k}): j\in J, k\in K\}$. Finally, $A_3$ contains one arc from every band vertex $w_{k}$ to the sink $t$, i.e. $A_{3}=\{(w_{k},t):  k\in K\}$.
By construction, $G(V,A)$ is bipartite and acyclic.
Each arc $a \in A$ is associated to a triple $(l_a,u_a,c_a)$, where $l_a,u_a$ are lower and upper bounds on the flow that can be sent on $a$ and $c_a$ is the cost of sending one unit of flow on $a$.
The values of the triples $(l_a,u_a,c_a)$ are set in the following way: $(0,1,0)$ when $a\in A_{1}$; $(0,1,-d_{ij}^{k}x_{j})$ when $a=(v_j,w_k)\in A_{2}$; $(l_k,u_k,0)$ when $a=(w_k,t)\in A_{3}$.
Finally, the amount of flow that must be sent trough the network from $s$ to $t$ is equal to $n$. We denote by $F_i$ the set of feasible integral flows of value $n$ of the min-cost flow instance $(G,c)_{x}^{i}$.
The \emph{cost} of an $(s,t)$-flow $f \in F_i$ is defined by $c(f)=\sum_{a\in A} c_{a}f_{a}$. We recall that an integral min-cost flow of value $n$ can be computed in polynomial time, using for example the successive shortest path algorithm \cite{AhMaOr93}.

We now prove that by solving the min-cost flow instance $(G,c)_{x}^{i}$ defined above, we obtain the maximum deviation for a constraint $i$ and a solution $x$.
\begin{theorem} \label{theorem:robCut}
Let $x \in \mathbb{R}_{+}^{n}$ and let $\mathcal{S}_{M}$ be a multi-band uncertainty set.
Moreover, let $(G,c)_{x}^{i}$ be the min-cost flow instance corresponding with $x$ and a constraint $i \in I$ of (LPP) and built according to the previously presented rules.

\noindent
The solution $x$ is robust feasible w.r.t. $\mathcal{S}_{M}$ for constraint $i$ if and only if
$$
\bar{a}_{i}' x - c^{*}_i(x) \leq b_{i}
$$
where $c^{*}_i(x)$ is the minimum cost of a flow of the instance $(G,c)_{x}^{i}$.
\end{theorem}

\begin{proof}
Before proceeding to the core of the proof, we need to make some observations.
Let $Y_{i}$ be the set of feasible solutions to problem (DEV01), i.e.:
$$
Y_{i} =
\{
    y_i \in \{0,1\}^{|J|\hspace{0.05cm}|K|}: y_i
    \mbox{ satisfies (\ref{DEV01_bounds})-(\ref{DEV01_singleBand})}
\}  .
$$

\noindent
Each of these solutions represents a feasible assignment of the coefficients of constraint $i$ to the deviation bands $K$. Moreover, let $d(x, y_i)$ be the total deviation associated with a vector $x \in \mathbb{R}_{+}^{n}$ and a feasible assignment $y_i \in Y_{i}$, i.e.:
$$
d(x, y_i) =
\sum_{j \in J} \sum_{k \in K} d^{k}_{ij} \hspace{0.1cm} x_{j} \hspace{0.1cm} y_{ij}^{k} .
$$

\noindent
We note that this is actually the objective function of problem (DEV01).

It is easy to verify that for each feasible assignment that distributes less than $n = |J|$ coefficients among the bands, there exists a feasible assignment that distributes \emph{all} the $n$ coefficients among the bands and entails at least the same total deviation, i.e.:
\begin{eqnarray*}
    &&\forall \hspace{0.1cm} y_i^{1} \in Y_{i}: \sum_{j \in J} \sum_{k \in K} y_{ij}^{k1} < n
    \\
    &&\exists \hspace{0.1cm} y_i^{2} \in Y_{i}:
    \left\{
                        \begin{array}{lll}
                            \sum_{j \in J} \sum_{k \in K} y_{ij}^{k2} = n ,
                            \\
                            \
                            \\
                            d(x, y_i^{2}) \geq d(x, y_i^{1}) .
                        \end{array}
                    \right.
\end{eqnarray*}

\noindent
Thanks to the last observation, in order to find a feasible assignment $y_i \in Y_{i}$ that maximizes the total deviation $d(x, y_i)$, we can concentrate attention on the subsets $\bar{Y}_{i} \subseteq Y_{i}$ of assignments which distribute all the $n$ coefficients among the bands:
$$
\bar{Y}_{i} =
\{
    y_i \in \{0,1\}^{|J|\hspace{0.05cm}|K|}: y_i
    \mbox{ satisfies (\ref{DEV01_bounds})-(\ref{DEV01_singleBand})
    and }
    \sum_{j \in J} \sum_{k \in K} y_{ij}^{k} = n
\} .
$$

\noindent
We now prove the lemma by showing that, given a solution $x \in \mathbb{R}_{+}^{n}$ and a constraint $i \in I$, there exists a one-to-one correspondence between
a feasible assignment $y_i \in \bar{Y}_{i}$ with total deviation $d(x, y_i)$ and an integral flow $f$ of value $n$ and cost $c(f) = - d(x, y_i)$ of the min-cost flow instance $(G,c)_{x}^{i}$.

Once that this correspondence is proved, we can solve the problem of finding an assignment with maximum deviation $DEV_i(x,\mathcal{S}_{M})$ by finding an integral flow of minimum cost of the instance $(G,c)_{x}^{i}$. Indeed the following chain of equations holds:
$$
DEV_i(x,\mathcal{S}_{M})
\hspace{0.1cm} = \hspace{0.1cm}
\max_{y_i \in \bar{Y}_{i}} d(x, y_i)
\hspace{0.1cm} = \hspace{0.1cm}
- \min_{y_i \in \bar{Y}_{i}} - d(x, y_i)
\hspace{0.1cm} = \hspace{0.1cm}
- \min_{f \in F_i} c(f)
\hspace{0.1cm} = \hspace{0.1cm}
- c^{*}_i(x) .
$$

\vspace{0.3cm}

\noindent
Given a feasible assignment $y_i \in \bar{Y}_{i}$ with total deviation $d(x, y_i)$, we construct an integral flow $f$ for the min-cost flow instance $(G,c)_{x}^{i}$, in the following way:
\begin{eqnarray}
&&
  f_{(s, v_j)} = 1
        \hspace{2.9cm} \forall j \in J
  \label{flowDef1}
  \\
&&
  f_{(v_j,w_k)} = 1
        \hspace{0.1cm} \Longleftrightarrow \hspace{0.1cm}
        y_{ij}^{k} = 1
        \hspace{0.6cm} \forall j \in J, k \in K
  \label{flowDef2}
  \\
&&
  f_{(w_k, t)} = \sum_{j \in J} y_{ij}^{k}
        \hspace{2.05cm} \forall k \in K .
  \label{flowDef3}
\end{eqnarray}

\noindent
It is clear that for each $j \in J, k \in K$ the flow components $f_{(s, v_j)}$, $f_{(v_j,w_k)}$ respect the capacity of the corresponding arcs. In the case of $f_{(w_k, t)}$, the capacity of arc $(w_k, t)$ is respected as (\ref{flowDef3}) holds and $l_k \leq \sum_{j \in J} y_{ij}^{k} \leq u_k$, $\forall k \in K$ by definition of $y_i$.
By (\ref{flowDef1}), the value of the flow $f$ is $\sum_{j \in J} f_{(s, v_j)} = |J| = n$ and its cost is:
$$
c(f)
\hspace{0.1cm} = \hspace{0.1cm}
\sum_{j \in J} \sum_{k \in K}
- d_{ij}^{k} \hspace{0.1cm} x_{j} \hspace{0.15cm} f_{(v_j,w_k)}
\hspace{0.1cm} = \hspace{0.1cm}
\sum_{j \in J} \sum_{k \in K} - d_{ij}^{k} \hspace{0.1cm} x_{j} \hspace{0.15cm} y_{ij}^{k}
\hspace{0.1cm} = \hspace{0.1cm}
- d(x, y_i) .
$$

\noindent
Concerning flow conservation, by definition of $y_i$, we know that each coefficient of the constraint $i$ is assigned to exactly one band, i.e. $\exists \mbox{ unique } k \in K: y_{ij}^{k} = 1$, $\forall j \in J$. As a consequence, flow conservation is respected in every vertex $v_j$ with $j \in J$.
In the case of a vertex $w_k$ with $k \in K$, we have:
$$
f_{(w_k, t)} = \sum_{j \in J} y_{ij}^{k} = \sum_{j \in J} f_{(v_j,w_k)} ,
$$
so the flow conservation is respected. Finally, in the case of the sink $t$, we have an ingoing flow of $\sum_{k \in K} f_{(w_k, t)} = \sum_{j \in J} \sum_{k \in K} y_{ij}^{k} = n$, which is thus equal to the value of the flow $f$.

\vspace{0.5cm}

\noindent
Given a feasible integral flow $f \in F_i$ of value $n$ and cost $c(f)$ of the min-cost flow instance $(G,c)_{x}^{i}$, we construct a deviation assignment $y_i \in \{0,1\}^{|J|\hspace{0.05cm}|K|}$, in the following way:
\begin{eqnarray}
    y_{ij}^{k} = 1
    \hspace{0.1cm} \Longleftrightarrow \hspace{0.1cm}
    f_{(v_j,w_k)} = 1
    \hspace{0.6cm} \forall j \in J, k \in K.
  \label{assignDef1}
\end{eqnarray}

\noindent
Since the value of $f$ is $n$ by definition and since each arc incident to a vertex $v_{j}$, $j \in J$ has lower and upper capacity bounds of 0 and 1, respectively, i) the flow on each arc $(s,v_{j})$ is unitary, ii) the ingoing flow of each vertex $v_{j}$ is unitary, and by flow conservation we have:
$$
\sum_{j \in J} \sum_{k \in K} f_{(v_j,w_k)} =  \sum_{j \in J} \sum_{k \in K} y_{ij}^{k} = n .
$$

\noindent
Moreover, by flow conservation and by the integrality of $f$, the outgoing flow of $v_{j}$ is unitary and is sent to a single node $w_k$. Therefore:
$$
\sum_{k \in K} f_{(v_j,w_k)} = \sum_{k \in K} y_{ij}^{k} = 1
$$

\noindent
Finally, by definition of lower and upper capacity bounds of an arc $f_{(w_k, t)}$ and by flow conservation, for every $k \in  K$ we have:
\begin{eqnarray*}
&&
l_k
\hspace{0.1cm} \leq \hspace{0.1cm}
f_{(w_k, t)}
\hspace{0.1cm} \leq \hspace{0.1cm}
u_k,
\\
&&
f_{(w_k, t)} = \sum_{j \in J} f_{(v_j,w_k)} = \sum_{j \in J} y_{ij}^{k}
\\
\
\\
&&
\Longrightarrow \hspace{0.2cm}  l_k
\hspace{0.1cm} \leq \hspace{0.1cm} \sum_{j \in J} y_{ij}^{k}
\hspace{0.1cm} \leq \hspace{0.1cm}
 u_k
\end{eqnarray*}

\noindent
By resuming all the properties of $y_i$ that we proved, it follows that $y_i \in \bar{Y}_{i}$ and the associated total deviation is equal to:
$$
- d(x, y_i)
\hspace{0.1cm} = \hspace{0.1cm}
\sum_{j \in J} \sum_{k \in K} - d_{ij}^{k} \hspace{0.1cm} x_{j} \hspace{0.15cm} y_{ij}^{k}
\hspace{0.1cm} = \hspace{0.1cm}
\sum_{j \in J} \sum_{k \in K} - d_{ij}^{k} \hspace{0.1cm} x_{j} \hspace{0.15cm} f_{(v_j,w_k)}
\hspace{0.1cm} = \hspace{0.1cm}
c(f) .
$$

\noindent
This concludes the proof of the existence of a one-to-one correspondence between the set $\bar{Y}_{i}$ of feasible deviation assignments and the set $F_i$ of integral flows and of the relation $
DEV_i(x,\mathcal{S}_{M})
=
- c^{*}_i(x)
$
between the optimal values of the maximum deviation problem (DEV01) and the min-cost flow problem of the instance $(G,c)_{x}^{i}$. The statement of the problem follows from this correspondence.
\qed
\end{proof}

\noindent
From Theorem \ref{theorem:robCut}, we can immediately derive the following corollary:

\begin{corollary} \label{corollary:robCut}
A solution $x \in \mathbb{R}_{+}^{n}$ is robust feasible w.r.t. $\mathcal{S}_{M}$ if and only if
$$
\bar{a}_{i}' x - c^{*}_i(x) \leq b_{i}
\hspace{1cm} \forall i \in I,
$$
where $c^{*}_i(x)$ is the minimum cost of a flow of the instance $(G,c)_{x}^{i}$.
\end{corollary}

\noindent
According to this corollary, we can test the robustness of a solution $x \in \mathbb{R}_{+}^{n}$
by computing an integral flow $f^{i}$ of minimum cost $c^{*}_i(x)$ in $(G,c)_{x}^{i}$ for every $i \in I$ (we recall that this can done in polynomial time).
If $\bar{a}' x - c^{*}_i(x) \leq b_{i}$ for every $i$,
then $x$ is robust feasible. Otherwise $x$ is not robust and there exists an index
 $i \in I$ such that $\bar{a}' x - c^{*}_i(x) > b_{i}$ and thus:
\begin{equation}\label{eq:robCut}
\sum_{j\in J} \bar{a}_{ij} \hspace{0.1cm} x_{j}
+ \sum_{j\in J} \sum_{k\in K} - d_{ij}^{k} \hspace{0.1cm} x_{j} \hspace{0.1cm} f^i_{(v_j,w_k)} \leq b_{i}
\end{equation}
is valid for the polytope of the robust
solutions and cuts off the solution $x$.

\section{Computational Study} \label{sec:compStudy}

In this section, we test our new modeling and solution approaches to Robust Optimization on a set of realistic instances of the \emph{Power Assignment Problem}, a problem arising in the design of wireless networks. In particular, we compare the efficiency of solving directly  the compact formulation (RLP) with that of a cutting plane method based on the robustness cuts presented in Section \ref{sec:rocuts}. In the case of the Bertsimas-Sim model, such comparison led to contrasting conclusions (e.g., \cite{FiMo08,KoKuRa11}).

\vspace{0.2cm}

\noindent
\textbf{The Power Assignment Problem.}
The \emph{Power Assignment Problem} (PAP) is the problem of dimensioning the power emission of each transmitter in a wireless network, in order to provide service coverage to a number of user, while minimizing the overall power emission. The PAP is particularly important in the (re)optimization of networks that are updated to new generation digital transmission technologies. For a detailed introduction to the PAP and the general problem of designing  wireless networks, we refer the reader to \cite{MaRoSm07,DA10,DAMaSa10,MaRoSm06}.

A classical LP formulation for the PAP can be defined by introducing the following elements: 1) a vector of non-negative continuous variables $p$ that represent the power emissions of the transmitters; 2) a vector $P^{\max}$ of upper bounds on $p$ that represent technology constraints on the maximum power emissions; 3) a matrix $A$ of the coefficients that represent signal attenuation (\emph{fading coefficients}) for each transmitter-user  couple; 4) a vector of r.h.s. $\delta$ (signal-to-interference thresholds) that represent the minimum power values that guarantee service coverage.
Under the objective of minimizing the overall power emission, the PAP can be written in the following matrix form:
$$
\min
\hspace{0.1cm}
\mbox{\textbf{1}}' p
\hspace{0.2cm} \mbox{ s.t. }
\hspace{0.1cm}
A p \geq \delta
, \hspace{0.2cm}
0 \leq p \leq P^{\max}
\hspace{1.0cm} (PAP)
$$

\noindent
where exactly one constraint $a'_i p\ge \delta_i$ is introduced for each user  $i$ to represent the corresponding service coverage condition.

Each entry of matrix $A$ is classically computed by a propagation model and takes into account many factors (e.g., distance between transmitter and receiver, terrain features). However, the exact propagation behavior of a signal cannot be evaluated and thus each fading coefficient is naturally subject to uncertainty.
Neglecting such uncertainty may provide unpleasant surprises in the final coverage plan, where devices may turn out to be uncovered for bad deviations affecting the fading coefficients (this is particularly true in hard propagation scenarios, such as dense urban fabric). For a detailed presentation of the technical aspects of propagation, we refer the reader to \cite{Ra01}.

Following the ITU recommendations (e.g., \cite{ITU02}), we assume that the fading coefficients are mutually independent random variables and that each variable is log-normally distributed.
The adoption of the Bertsimas-Sim model would provide only a rough representation of the deviations associated with such distribution. We thus adopt the multi-band uncertainty model to obtain a more refined representation of the fading coefficient deviations.
In what follows, we denote the Bertsimas-Sim and the multi-band uncertainty model by (BS) and (MB), respectively.

\vspace{0.2cm}

\noindent
\textbf{Computational Results.}
In this computational study, we consider realistic instances corresponding to region-wide networks that implement the Terrestrial Digital Video Broadcasting technology (DVB-T) \cite{ITU02} and were taken as reference for the design of the new Italian DVB-T national network.
The uncertainty set is built taking into account the ITU recommendations \cite{ITU02} and discussions with our industrial partners in past projects about wireless network design.
Specifically, we assume that each fading coefficient follows a log-normal distribution with mean provided by the propagation model and standard deviation equal to 5.5 dB~\cite{ITU02}.
In our test-bed, the (MB) uncertainty set of a generic fading coefficient $a_{ij}$ is
constituted by 3 negative and 3 positive deviations bands (i.e., $K = 6$). Each band has a width equal to the 5$\%$ of the nominal fading value $\bar{a}_{ij}$. Thus the maximum allowed deviation is +/- $0.15 \cdot \bar{a}_{ij}$. For each constraint $i$, the bounds $l_k,u_k$ on the number of deviations are defined considering the cumulative distribution function of a log-normal random variable with standard deviation 5.5 dB.
The (BS) uncertainty set of each constraint considers the same maximum deviation of (MB) and the maximum number of deviating coefficients is $\Gamma = \lceil0.8 \cdot u^{\max}\rceil$, where $u^{\max} = \max \{ u_k : k \in K \backslash \{0\}\}$. This technically reasonable assumption on $\Gamma$ ensures that (BS) does not dominate (MB) a priori.

The computational results are reported in Table \ref{tab:results}. The tests were performed on a Windows machine with 1.80 GHz Intel Core 2 Duo processor and 2 GB RAM. All the formulations are implemented in C++ and solved by IBM ILOG Cplex 12.1, invoked by ILOG Concert Technology 2.9.
We considered 15 instances of increasing size corresponding to realistic DVB-T networks.  The first column of Table \ref{tab:results} indicates the ID of the instances. Columns $|I|,|J|$ indicate the number of variables and constraints of the problem, corresponding to the number of user devices and transmitters of the network, respectively. We remark that the coefficient matrices tend to be sparse, as only a (small) fraction of the transmitters is able to reach a user device with its signals. Columns $|I^{+}|,|J^{+}|$ indicate the number of additional variables and constraints needed in the compact robust counterpart (RLP). Columns PoR\% report \emph{the Price of Robustness} (PoR), i.e. the deterioration of the optimal value required to guarantee robustness. In particular, we consider the percentage increase of the robust optimal value w.r.t. the optimal value of the nominal problem, in the multi-band case (PoR\% (MB)) and in the Bertsimas-Sim case (PoR\% (BS)). Column $\Delta t\%$ reports the percentage increase of the time required to compute the robust optimal solution under (MB) by using the cutting plane method presented in Section \ref{sec:rocuts} w.r.t. the time needed to solve the compact formulation (RLP).
Finally, column Protect\% is a measure of the protection offered by the robust optimal solution and is computed in the following way: for each instance, we generate 1000 realizations of the uncertain coefficient matrix and we then compute the percentage of realizations in which the robust optimal solution is feasible. This is done for both (MB) and (BS).

Looking at Table \ref{tab:results}, the first evident thing is that the dimension of the compact robust counterpart under (MB) is much larger than that of the nominal problem. However, this is not an issue for Cplex, as all instances are solved within one hour and in most of the cases the direct solution of (RLP) takes less time than the cutting plane approach ($\Delta t\% < 0$). Anyway, for the instances of greater dimension the cutting plane approach becomes competitive and may even take less time ($\Delta t\% > 0$). Concerning the PoR, we note that under (MB) imposing robustness leads to a sensible increase in the overall power emission, that is anyway lower than that of (BS) in all but two cases. On the other hand, such increase of (MB) is compensated by a very good 90\% protection on average.
In the case of the PAP, (MB) thus seems convenient to model the log-normal uncertainty of fading coefficients, guaranteeing good protection at a reasonable price. Moreover, though (BS) offers higher protection for most instances, it is interesting to note that the increase of Protect\% of (BS) w.r.t (MB) is lower than the corresponding increase of PoR\% of (BS) w.r.t (MB).

\begin{table}[htbp]
\caption{Overview of the computational results} \label{tab:results}
{\small
\begin{center}
\begin{tabular}{|ccc||ccccc|ccccc|ccccccccccc|}
  \hline
  & & &
  & & & & &
  & & & & &
  & \raisebox{-0.5ex}{PoR\%} & & \raisebox{-0.5ex}{PoR\%} & & & &
  \raisebox{-0.5ex}{Protect\%} & & \raisebox{-0.5ex}{Protect\%} &
  \\
  & \raisebox{1.5ex}{$\mbox{ ID }$} & &
  %
  \raisebox{1.5ex}{$\mbox{ }$} &
    \raisebox{1.5ex}{$|I|$} &
  \raisebox{1.5ex}{$\mbox{ }$} &
    \raisebox{1.5ex}{$|J|$} &
  \raisebox{1.5ex}{$\mbox{ }$} &
  %
  \raisebox{1.5ex}{$\mbox{ }$} &
    \raisebox{1.5ex}{$|I^{+}|$} &
  \raisebox{1.5ex}{$\mbox{ }$} &
    \raisebox{1.5ex}{$|J^{+}|$} &
  \raisebox{1.5ex}{$\mbox{ }$} &
  %
  \raisebox{1.5ex}{$\mbox{ }$} &
    \raisebox{0.5ex}{(MB)} &
  \raisebox{1.5ex}{$\mbox{ }$} &
    \raisebox{0.5ex}{(BS)} &
  \raisebox{1.5ex}{$\mbox{ }$} &
    \raisebox{1.5ex}{$\Delta$t\%} &
  \raisebox{1.5ex}{$\mbox{ }$} &
    \raisebox{0.5ex}{(MB)} &
  \raisebox{1.5ex}{$\mbox{ }$} &
    \raisebox{0.5ex}{(BS)} &
  \\
  \hline
  \hline
  & D1 & &
  & 95 &   & 153 &   &
  & 3519 &   & 10098 &   &
  & 8.3 &  & 10.1 &  & -18.7 &  & 88.20 &
  & 92.53 &
  \\
  & D2 & &
  & 103 &   & 197 &   &
  & 4728 &   & 14184 &   &
  & 7.2 &  & 9.4 &  & -19 &  & 91.35 &
  & 92.47 &
  \\
  & D3 & &
  & 105 &   & 322 &   &
  & 7406 &   & 21252 &   &
  & 6.8 &  & 8.8 &  & -16.9 &  & 93.12 &
  & 96.40 &
  \\
  & D4 & &
  & 105 &   & 473 &   &
  & 10406 &   & 28380 &   &
  & 7.4 &  & 7.2 &  & -15.1 &  & 92.08 &
  & 91.42 &
  \\
  & D5 & &
  & 108 &   & 569 &   &
  & 13087 &   & 37554 &   &
  & 9.2 &  & 11.4 &  & -13.6 &  & 89.23 &
  & 90.29 &
  \\
  & D6 & &
  & 157 &   & 1088 &   &
  & 27200 &   & 84864 &   &
  & 6.6 &  & 9.1 &  & -6.2 &  & 85.46 &
  & 87.55 &
  \\
  & D7 & &
  & 165 &   & 1203 &   &
  & 31278 &   & 101052 &   &
  & 7.1 &  & 9.5 &  & -4.9 &  & 87.91 &
  & 89.16 &
  \\
  & D8 & &
  & 171 &   & 1262 &   &
  & 32812 &   & 106008 &   &
  & 8.7 &  & 10.8 &  & -4.1 &  & 89.40 &
  & 93.08 &
  \\
  & D9 & &
  & 178 &   & 1375 &   &
  & 35750 &   & 115500 &   &
  & 9.6 &  & 10.2 &  & -2.8 &  & 90.11 &
  & 91.90 &
  \\
  & D10 & &
  & 180 &   & 1448 &   &
  & 39096 &   & 130320 &   &
  & 7.9 &  & 9.6 &  & -1.7 &  & 91.54 &
  & 95.32 &
  \\
  & D11 & &
  & 180 &   & 1661 &   &
  & 46058 &   & 159456 &   &
  & 7.2 &  & 9.5 &  & 0.6 &  & 94.77 &
  & 96.70 &
  \\
  & D12 & &
  & 181 &   & 1779 &   &
  & 49812 &   & 170784 &   &
  & 7.5 &  & 10.1 &  & 1.8 &  & 88.22 &
  & 90.16 &
  \\
  & D13 & &
  & 183 &   & 1853 &   &
  & 53737 &   & 189006 &   &
  & 8.1 &  & 10.3 &  & 3.3 &  & 91.34 &
  & 92.21 &
  \\
  & D14 & &
  & 183 &   & 1940 &   &
  & 56260 &   & 197880 &   &
  & 10.3 &  & 9.7 &  & 3.1 &  & 86.50 &
  & 85.18 &
  \\
  & D15 & &
  & 185 &   & 2183 &   &
  & 63307 &   & 222666 &   &
  & 8.4 &  & 10.8 &  & 4.1 &  & 91.09 &
  & 92.70 &
  \\
  \hline
\end{tabular}
\end{center}
} 
\end{table}

\section{Conclusions and Future Work}  \label{sec:end}

In this work, we presented new theoretical results about multi-band uncertainty in Robust Optimization. Surprisingly, this natural and practically very relevant extension of the classical single band model by Bertsimas and Sim has attracted very little attention and we have thus started to fill this theoretical gap. We showed that, under multi-band uncertainty, the robust counterpart of an LP is linear and compact and that the problem of separating a robustness cut can be formulated as a min-cost flow problem and thus be solved efficiently. Tests on realistic network design instances showed that our new approach performs very well, thus encouraging further investigations. Future research will focus on refining the cutting plane method and enlarging the computational experience to other relevant real-world problems.

\end{document}